\documentclass[12pt]{article}
\usepackage{amsmath}
\usepackage{amsthm}
\usepackage{amssymb}
\usepackage{authblk}
\usepackage{dsfont}
\usepackage{faktor}
\usepackage{graphicx}
\usepackage{enumerate}
\usepackage{xcolor}
\usepackage{color}
\usepackage[utf8]{inputenc}
\usepackage[T1]{fontenc}
\usepackage[
    inner=2.3cm,
    outer=2.3cm,
    top=2.7cm,
    marginparwidth=2.5cm,
    marginparsep=0.1cm
    ]
    {geometry}
\usepackage[pdffitwindow=false,
            plainpages=false,
            pdfpagelabels=true,
            pdfpagemode=UseOutlines,
            pdfpagelayout=SinglePage,
            bookmarks=false,
            colorlinks=true,
            hyperfootnotes=false,
            linkcolor=blue,
            urlcolor=blue!30!black,
            citecolor=green!50!black]{hyperref}
    
\newcommand{\tor}{\mathds{T}^2}
\newcommand{\R}{\mathds{R}}
\newcommand{\N}{\mathds{N}}

\newcommand{\Z}{\mathds{Z}}

\newcommand{\SL}[1]{\textup{SL}_{#1}(\R)}
\newcommand{\norm}[1]{\left\lVert#1\right\rVert}

\newcommand{\Id}{\textup{Id}}

\newcommand{\rvline}{\hspace*{-\arraycolsep}\vline\hspace*{-\arraycolsep}}
\newenvironment{revision}{\color{black}}{\color{black}}

\newenvironment{nalign}{
    \begin{equation}
    \begin{aligned}
}{
    \end{aligned}
    \end{equation}
    \ignorespacesafterend
}

\newtheorem{theorem}{Theorem}[section]

\newtheorem{lemma}[theorem]{Lemma}
\newtheorem{corollary}[theorem]{Corollary}
\newtheorem{definition}[theorem]{Definition}

\theoremstyle{definition}

\theoremstyle{definition}
\newtheorem{remark}[theorem]{Remark}

\title{Continuous-time extensions of discrete-time cocycles}
\author[1]{Robin Chemnitz}
\author[1,3]{Maximilian Engel}
\author[1,2]{P\'eter Koltai}

\affil[1]{Institute of Mathematics, Freie Universit\"at Berlin, 14195 Berlin, Germany}
\affil[2]{Department of Mathematics,  University of Bayreuth, 95440 Bayreuth, Germany}
\affil[3]{KdV Institute for Mathematics, University of Amsterdam, 1098 XG Amsterdam, The Netherlands}

\date{}

\begin{document}

\maketitle

\begin{abstract}
    We consider linear cocycles taking values in $\textup{SL}_d(\mathbb{R})$ driven by homeomorphic transformations of a smooth manifold, in discrete and continuous time. We show that any discrete-time cocycle can be extended to a continuous-time cocycle, while preserving its characteristic properties. We provide a necessary and sufficient condition under which this extension is canonical in the sense that the base is extended to an associated suspension flow and that 
    \begin{revision}the discrete-time cocycle is recovered as the time-1 map of the continuous-time cocycle.\end{revision} 
    Further, we refine our general result for the case of (quasi-)periodic driving. 
    \begin{revision}
        We use our findings to construct a non-uniformly hyperbolic continuous-time cocycle in $\SL{2}$ over a uniquely ergodic driving.
    \end{revision}
\end{abstract}


\section{Introduction}

One of the most commonly studied types of linear cocycles are those taking values in the special linear group $\SL{d}$. Motivated by questions from random matrix, spectral and non-uniformly hyperbolic theory, there has been considerable effort to understand and categorise their Lyapunov exponents and Oseledets splittings \cite{Avila2009, Avila2011, Knill1990, Oseledec1968, SimonTaylor1985, Viana2008, Young1995}.
Classically, these cocycles are studied over a discrete-time driving while continuous-time cocycles in $\SL{d}$ have received much less attention in the past, even though they are of high interest for the study of nonautonomous ODEs. 
\begin{revision}In particular, continuous-time cocycles are often studied in terms of their time-1 maps. We address the question which discrete-time cocycles can occur as the time-1 map of a continuous-time cocycle.\end{revision}

In our setting the discrete-time driving is given by a homeomorphism on a compact, smooth manifold. Such a map can be extended to continuous time via the classical \emph{suspension flow} construction~\cite[\S 1.3]{krengel2011ergodic}. Although one could extend the \emph{whole} skew-product associated with a discrete-time cocycle by creating an associated suspension flow,
the cocycle property would be lost in this process.
Given a discrete-time cocycle we study the question whether there is a continuous-time cocycle over a suspension flow whose time-1 map coincides with the discrete-time cocycle. \begin{revision}We call such a continuous-time cocycle a \emph{canonical extension}, cf.~Definition~\ref{def:cano}.\end{revision} The ``autonomous'' but infinite-dimensional version of our question, namely where bounded linear operators are extended to $C_0$-semigroups, was considered in~\cite{eisner2009embedding}. A special case that yields particularly interesting results is that of an (irrational) circle rotation as the driving system. \begin{revision}For these systems, we not only study whether a given discrete-time cocycle has a canonical extension but also how cocycles without canonical extensions need to be modified such that a canonical extension is possible.\end{revision} Our main results are as follows:
\begin{enumerate}
    \item Theorem \ref{thm:nullhomotopic_extension} provides an equivalent, homotopy-based characterisation of the canonical extendability of discrete-time cocycles in $\SL{d}$ over arbitrary driving.
    \item Corollary~\ref{crl:n_geq_3_extension} shows that for (quasi-)periodic driving and \begin{revision}fibre-dimension $d\ge 3$, a canonical
    extension is always possible, after possibly modifying the base space.\end{revision}
    \item Corollary~\ref{crl:n_2_extension} paves the way to extend any discrete-time cocycle over (quasi-)periodic driving in $\SL{2}$ to continuous time upon adding an auxiliary dimension.
\end{enumerate}

\begin{revision}
    The statement of Theorem \ref{thm:nullhomotopic_extension} is of interest in both directions. Not only does the theorem show which discrete-time cocycles can be extended to continuous time, but also provides a characterisation of which discrete-time cocycles can occur as time-1 maps of continuous-time cocycles.
\end{revision}
We would like to emphasise that while in autonomous dynamical system theory an extension from discrete time to continuous time requires in general an added auxiliary dimension, our results show that for 
a (quasi-)periodically driven system
an additional dimension is only needed in the case \begin{revision}of fibre-dimension\end{revision} $d=2$. For $d\ge 3$, any discrete-time cocycle can be extended to a continuous-time cocycle of the same dimensionality\begin{revision}, after possibly modifying the base space. We note that this modification leaves the dynamical characteristics of the cocycle, like its Lyapunov exponents, Oseledets spaces and (non-)uniform hyperbolicity unchanged, cf.~Remark  \ref{rem:preserve_hyper}.\end{revision}

\begin{revision}
    In addition to the general theory, we will construct a discrete-time cocycle in $\SL{2}$ over an irrational circle rotation that is non-uniformly hyperbolic and has a canonical extension to continuous time. This shows that there is a continuous-time cocycle in $\SL{2}$ over a uniquely ergodic driving (i.e.~with unique invariant probability measure) which is non-uniformly hyperbolic. Prior to the present work, the authors were not aware of any continuous-time cocycle with these properties. The existence of such a cocycle stands in contrast to the uniform ergodic theorem of Oxtoby \cite{oxtoby1952ergodic}, which states that the Birkhoff ergodic theorem for continuous functions and a uniquely ergodic system holds uniformly. Hence, there is no analogous uniform multiplicative ergodic theorem for cocycles with uniquely ergodic driving, neither in discrete nor in continuous time. 
\end{revision}

The remainder of the paper is structured as follows. Section \ref{sec:preliminaries} introduces continuous cocycles in discrete and continuous time, the MET and homotopies. In Section \ref{sec:cont_time_extensions}, we study continuous-time extensions of discrete-time cocycles and characterise under which assumptions they exist. In Section \ref{sec:construction}, these results are applied to \begin{revision}construct a continuous-time cocycle on $\SL{2}$ over a uniquely ergodic driving that is non-uniformly hyperbolic.\end{revision}

\section{Preliminaries}\label{sec:preliminaries}
\subsection{Cocycles in $\SL{d}$}
We briefly introduce continuous cocycles in $\SL{d}$ in discrete time for which we aim to find continuous-time extensions. Let $\Theta \subset \R^n$ be a compact, smooth manifold with or without boundary \cite{lee2012smooth}. Consider a homeomorphism $\phi: \Theta \to \Theta$. Let $\mu$ be a $\phi$-ergodic probability measure on~$\Theta$. By the theorem of Krylov--Bogolyubov, such a measure $\mu$ exists but is in general not unique.

A continuous (not to be confused with continuous-time) cocycle in $\SL{d}$ over $\phi$ is generated by a continuous map
\begin{equation}
    \label{eq:cocycle}
    A: \Theta \to \SL{d},
\end{equation}
where $\SL{d}$ denotes the special linear group of degree $d$, i.e.~the group of $d\times d$ matrices with determinant~$1$. We also use the notation $A\in C(\Theta, \SL{d})$.
The values of $A$ are denoted by~$A_\theta$. For simplicity, we restrict ourselves to cocycles in~$\SL{d}$, since any cocycle in $\textup{GL}_d(\R)$ with positive determinant can be reduced to one in $\SL{d}$ via rescaling. Since the driving $\phi$ is invertible, we can define the $n$-step map of $A$ in two-sided time, i.e.~for $n\in \Z$
\begin{equation}\label{eq:n_step_cocycle}
    A_\theta^n := \begin{cases} A_{\phi^{n-1}\theta}\cdot\hdots \cdot A_{\phi\theta}\cdot A_\theta, \quad & n>0,\\
    \Id, &n=0,\\
     A_{\phi^{n}\theta}^{-1} \cdot \hdots \cdot A_{\phi^{-2}\theta}^{-1} \cdot A_{\phi^{-1}\theta}^{-1}, \quad &n<0.
    \end{cases}
\end{equation}
With this definition, $A$ satisfies the cocycle property
\begin{equation}
    A^{m+n}_\theta = A^{n}_{\phi^m \theta} \cdot A^m_\theta, \quad \forall n,m\in \Z,
\end{equation}
for all $\theta \in \Theta$. In particular, we have the identity
\begin{equation}\label{eq:negative_time_identity}
    A^{-n}_\theta = (A^n_{\phi^{-n} \theta})^{-1}, \quad \forall n\in \N.
\end{equation}
Since a discrete-time cocycle is uniquely represented by its generator, we use the same symbol $A$ for both. Note that $A_\theta^n$ never refers to the matrix power. 

The canonical way of extending a discrete-time map $\phi:\Theta \to \Theta$ to continuous time is via a suspension flow construction \cite{brin2002introduction}. Consider the so-called \emph{mapping torus}
\begin{equation}
    \Theta_1 = (\Theta \times [0,1] )/_\sim,
\end{equation}
over the equivalence relation $(\theta, 1) \sim (\phi(\theta), 0)$. On this space, we define the continuous-time flow
\begin{nalign}\label{eq:suspension_flow}
    \phi^t: \hspace{3mm}\Theta_1 &\to \Theta_1 \\
    (\theta, r) &\mapsto (\phi^{\lfloor r+t \rfloor}\theta ,r+t \mod 1),
\end{nalign}
where $\phi^{\lfloor r+t \rfloor}:\Theta \to \Theta$ is the discrete time map $\phi$ applied $\lfloor r+t \rfloor$ times. The flow $\phi^t$ moves points in $\Theta_1$ up in the second coordinate until they hit the ceiling $\Theta \times \{1\}$. Since the ceiling of $\Theta_1$ is associated to the bottom $\Theta_1 \times \{0 \}$ via the equivalence relation, this defines the flow $\phi^t$ for all times $t\geq 0$. The flow~\eqref{eq:suspension_flow} is well-defined for all $t\in \R$ such that $\phi^t$ is an invertible flow. Note that the time-1 map of $\phi^t$ coincides with the discrete-time map $\phi :\Theta \to \Theta$ in the sense that~$\phi^1(\theta, r) = (\phi(\theta), r)$. Naturally, $\phi^t$ is ergodic with respect to the measure $\mu \otimes \lambda$, where $\lambda$ denotes the Lebesgue measure on~$[0,1]$. 

We define a continous-time cocycle via a generator represented by a continuous map
\begin{equation}
    G: \Theta_1 \to  \mathfrak{sl}_d(\R),
\end{equation}
where $ \mathfrak{sl}_d(\R)$ is the special linear Lie algebra of real $d\times d$ matrices with trace 0. For each $(\theta, r) \in \Theta_1$ this generator defines a nonautonomous ordinary differential equation (ODE) in $\R^d$
\begin{equation}\label{eq:cont_time_ODE}
    \partial_t x(t) = G(\phi^t (\theta, r)) x(t).
\end{equation}
We denote the fundamental solution to this ODE by $\Phi^t_{(\theta, r)}$, i.e.
\begin{equation}
    \Phi_{(\theta, r)}^tx(0) = x(t), \quad \forall t\ge 0.
\end{equation}
Equivalently, $\Phi$ is described as the unique solution to the nonautonomous ODE in $\SL{d}$
\begin{equation}\label{eq:cont_time_cocycle_ODE}
    \Phi^0_{(\theta, r)}=\Id, \qquad \partial_t \Phi^t_{(\theta, r)} = G(\phi^t(\theta, r)) \cdot \Phi^t_{(\theta, r)}.
\end{equation}
Since $\phi^t$ is invertible, we can solve the ODE for both positive and negative time. This defines a continuous map 
\begin{align}
\begin{split}
    \Phi : \Theta_1 \times \R &\to \SL{d}, \\
    ((\theta, r), t) &\mapsto \Phi^t_{(\theta, r)}.
\end{split}
\end{align}
 Liouville's formula guarantees that $\Phi_{(\theta, r)}^t\in \SL{d}$ since $G(\theta, r)\in  \mathfrak{sl}_d(\R)$ for all ${(\theta, r)} \in \Theta_1$. We call $\Phi$ the cocycle generated by $G$. Indeed, one can check that $\Phi$ satisfies the cocycle property
\begin{equation}
    \Phi_{(\theta, r)}^{s+t}=\Phi_{\phi^s {(\theta, r)}}^t\cdot \Phi_{(\theta, r)}^s, \quad \forall s,t\in \R,
\end{equation}
for all ${(\theta, r)} \in \Theta_1$. 

\subsection{The multiplicative ergodic theorem}\label{sec:MET}
The central tool in the study of linear cocycles is the multiplicative ergodic theorem (MET) which characterises Lyapunov exponents and Oseledets spaces. All notions and theorems introduced in this section hold for discrete and continuous time alike. Hence, let $\mathfrak{T}$ be either $\Z$ or $\R$, and $\Omega$ be $\Theta$ or $\Theta_1$ respectively. The driving $\phi^t:\Omega \to \Omega$ for $t\in \mathfrak{T}$ is well-defined for both the discrete and continuous-time case.

Consider a continuous cocycle
\begin{align}
\begin{split}
    \Phi : \Omega \times \mathfrak{T} &\to \SL{d}, \\
    (\omega, t) \hspace{1mm} &\mapsto \Phi^t_\omega.
\end{split}
\end{align}
A classical result is the existence of the so-called Furstenberg--Kesten limit \cite{FurstenbergKesten1960},
\begin{equation}\label{eq:FK}
    \Lambda(\omega) = \lim_{t\to \infty} \frac{1}{t} \log \norm{\Phi_\omega^t}.
\end{equation}
There is a $\phi$-invariant set $\Omega'$ of full $\mu$-measure on which this limit exists and takes a constant value \begin{revision}$\Lambda^*$\end{revision}.

For $\omega \in \Omega$ and $x\in \R^d$, we define the upper Lyapunov exponent of $x$ in $\omega$
\begin{equation}
    \overline{\lambda}(\omega, x) = \limsup_{t\to \infty} \frac{1}{t} \log \norm{\Phi^t_\omega x}.
\end{equation}
In the case that the limit is exact we write $\lambda(\omega, x)$ and call it an exact Lyapunov exponent or simply Lyapunov exponent. We also define the lower backwards Lyapunov exponent of $x$ in $\omega$
\begin{equation}
    \underline{\lambda}^-(\omega, x) =\liminf_{t\to -\infty} \frac{1}{t} \log \norm{\Phi^t_\omega x}.
\end{equation}
We note that the fraction $\frac{1}{t}$ is negative. In the case that the limit is exact we write $\lambda^-(\omega, x)$. We formulate a version of the multiplicative ergodic theorem for continuous cocycles over an ergodic base (cf.~e.g.~\cite{Viana2014}).
\begin{theorem} \label{thm:MET}
    There are a $\phi$-invariant set $\Omega' \subset \Omega$ of full $\mu$-measure and numbers $\lambda_1 > \hdots > \lambda_k$\begin{revision}, where $\lambda_1=\Lambda^*$,\end{revision} such that for every $\omega\in \Omega'$ there is a measurable decomposition
    \begin{equation}
        \R^d = E_1(\omega) \oplus \cdots \oplus E_k(\omega),
    \end{equation}
    called \emph{Oseledets splitting}, with the following properties
    \begin{enumerate}[(i)]
        \item $\Phi^t_\omega E_i(\omega)= E_i(\phi^t \omega), \quad t\in \mathfrak{T}$;
        \item $\lambda(\omega, x) = \lambda^-(\omega, x)=\lambda_i, \quad x\in E_i(\omega)$.
    \end{enumerate}
\end{theorem}

The MET is usually stated for measurable cocycles without the assumption of continuity. In that case, one needs the additional assumption
\begin{equation}
    \log\norm{\Phi^{\pm 1}}\in L^1(\Omega).
\end{equation}
Since continuous functions on a compact space, like $\Omega$, are bounded, this condition is fulfilled for continuous cocycles.

\begin{revision}
    Lastly, we define the concept of (non-)uniform hyperbolicity. A cocycle in $\SL{d}$ is called \emph{hyperbolic}, if $\lambda_1=\Lambda^*>0$. 
    In the following definition, continuity of subspaces is meant with respect to the Grassmanian topology (cf.~e.g.~\cite[Example 1.15]{lee2012smooth}).
    \begin{definition}\label{def:uniform_hyper}
        A cocycle is called \emph{uniformly hyperbolic} if there are constants $C>0$, $\lambda>0$ and for each $\omega \in \Omega$ there is a continuous splitting 
        \begin{equation}
            \R^d = U(\omega) \oplus S(\omega),
        \end{equation}
         with the following properties
        \begin{enumerate}[(i)]
            \item $\Phi^t_\omega U(\omega)= U(\phi^t \omega)$ and $\Phi^t_\omega S(\omega)= S(\phi^t \omega), \quad t\in \mathfrak{T}$;
            \item $\norm{\Phi^t_\omega u} \geq C e^{\lambda t}, \qquad u\in U(\omega), \: t\in \mathfrak{T}_{\geq 0}$;
            \item $\norm{\Phi^t_\omega s} \leq C^{-1} e^{-\lambda t}, \qquad s\in S(\omega), \: t\in \mathfrak{T}_{\geq 0}$.
        \end{enumerate}
    \end{definition}
    A cocycle that is hyperbolic but not uniformly hyperbolic is called \emph{non-uniformly hyperbolic}. 

    For $d=2$, the (non-)uniform hyperbolicity of a cocycle in $\SL{2}$ is directly linked to its Oseletets splitting. If $\lambda_1>0$ (and thereby $\lambda_2=-\lambda_1 < 0$), Theorem \ref{thm:MET} provides a splitting $\R^2= E_1(\omega) \oplus E_2(\omega)$ for $\mu$-a.e.~$\omega \in \Omega$. If the cocycle is uniformly hyperbolic, the stable and unstable bundles $U(\omega)$ and $S(\omega)$ are given by $E_1(\omega)$ and $E_2(\omega)$. Hence, a hyperbolic cocycle in $\SL{2}$ is uniformly hyperbolic if and only if the Oseledets splitting is defined everywhere, it is continuous and satisfy the uniform growth/decay condition of Definition \ref{def:uniform_hyper}. In the case of uniquely ergodic driving, the last condition is redundant.
\end{revision}

\subsection{Homotopies}\label{sec:homotopies}
In the following section, we need some tools from topology which we briefly recap. For a formal introduction to homotopies and the fundamental group we refer to \cite{hatcher2000algebraic}. 

Given two topological spaces $\Theta$ and $M$, we study continuous maps $\psi\in C(\Theta, M)$. Given two maps $\psi_0, \psi_1 \in C(\Theta,M)$, a \emph{homotopy} $h$ from $\psi_0$ to $\psi_1$ is a continuous map $h:\Theta \times [0,1] \to M$ with
\begin{align}
\begin{split}
    h(\theta, 0) &= \psi_0(\theta),\\
    h(\theta, 1) &= \psi_1(\theta).
\end{split}
\end{align}
If such a map exists, we say that $\psi_0$ and $\psi_1$ are \emph{homotopic}, written $\psi_0 \simeq \psi_1$. When $M$ is a smooth manifold and $h$ is smooth, we call $h$ a \emph{smooth homotopy}. If $h$ is only smooth in the second variable, i.e.~$ t\mapsto h(\theta, t)$ is smooth for every $\theta\in \Theta$, we say that $h$ is \emph{smooth in $t$}. A homotopy can be expressed as a map $h:[0,1]\to C(\Theta, M)$ with $h(0)=\psi_0$ and $h(1)=\psi_1$. \begin{revision}When $\Theta$ is a compact metric space and $M$ is a metric space, this map is continuous with respect to the supremum norm on $C(\Theta, M)$.\end{revision} The homotopy relation $\simeq$ defines an equivalence relation. For $\psi\in C(\Theta, M)$ we denote its equivalence class by $[\psi]$.

Let us assume that $M$ is connected. Let $x_0\in M$ be any point and consider the constant map $\psi_0 \in C(\Theta, M)$ defined by $\psi_0(\theta)=x_0$. We call the equivalence class $[\psi_0]$ the \emph{null-class}. Note that any constant map lies in this equivalence class. We call maps $\psi\in [\psi_0]$ that are elements of the null-class \emph{nullhomotopic}. Hence, a map $\psi\in C(\Theta,M)$ is nullhomotopic if and only if it is homotopic to a constant map.

A special case, which is the most well-known application of homotopy, is $\Theta= S^1$. In that case, one can equip the set of equivalence classes with a group structure forming the so-called \emph{fundamental group} of $M$. \begin{revision}For $\Theta=S^n$, for any $n\in \N$, the equivalence classes can be equipped with a group structure as well, yielding the higher order homotopy groups of $M$ (cf.~e.g.~\cite{hatcher2000algebraic}).\end{revision} We call continuous maps $\psi\in C(S^1, M)$ \emph{loops}. Assume that $M$ is connected and fix a basepoint $x_0\in M$. For each equivalence class $[\psi]$, we can choose a representative $\psi^* \in [\psi]$ which satisfies $\psi^*(0) = x_0$. This allows us to define a group operation on the space of equivalence classes via concatenation. For two equivalence classes $[\psi_1]$, $[\psi_2]$, define the group operation by
\begin{equation}\label{eq:concatenation}
    [\psi_1]\circ [\psi_2] = [\psi], \qquad \psi(\theta) = \begin{cases}
        \psi_1(2\theta), &\theta \in [0, \frac{1}{2}), \\
        \psi_2(2\theta - 1), \quad & \theta \in [\frac{1}{2}, 1].
    \end{cases}
\end{equation}
This operation is indeed a group operation of the space of equivalence classes with the null-class as the neutral element. The resulting group is called the \emph{fundamental group} of $M$ and is denoted by $\pi_1(M)$.

\section{Continuous-time extensions}\label{sec:cont_time_extensions}
\subsection{General driving}\label{sec:general driving}

Given a continuous-time cocycle $\Phi$ over the driving $\phi^t:\Theta_1\to \Theta_1$, the time-1 map defines a discrete-time cocycle $A$ over the driving $\phi:\Theta \to \Theta$ via
\begin{equation}
    A_\theta = \Phi^1_{(\theta, 0)}.
\end{equation}
Due to the identity $\phi^n(\theta, 0)=(\phi^n \theta, 0)$, the $n$-step map of $A$ for $n\in \N$ is given by
\begin{align}\label{eq:disc_cont_comp_1}
\begin{split}
    A_{\theta}^n &= A_{\phi^{n-1} \theta}\cdot\hdots \cdot A_{\phi \theta}\cdot A_{\theta}\\
    &= \Phi^1_{\phi^{n-1}(\theta, 0)}\cdot \hdots \cdot \Phi^1_{\phi^1(\theta, 0)} \cdot \Phi^1_{(\theta, 0)} \\
    &= \Phi^n_{(\theta, 0)}.
\end{split}
\end{align}
For negative time, we can use the cocycle property of $\Phi$ and $A$, in particular identity (\ref{eq:negative_time_identity}), to compute
\begin{equation}\label{eq:disc_cont_comp_2}
    A^{-n}_{\theta} = (A^n_{\phi^{-n}\theta})^{-1} = (\Phi^n_{\phi^{-n}(\theta, 0)})^{-1} = \Phi^{-n}_{(\theta, 0)}.
\end{equation}
Hence, every continuous-time cocycle defines a discrete-time cocycle. \begin{revision}The question arises which discrete-time cocycles can occur as the time-1 maps of a continuous-time cocycle. Equivalently, for which discrete-time cocycles $A$ over $\Theta$ can we find a continuous-time cocycle $\Phi$ over $\Theta_1$ that has $A$ as its time-1 map?\end{revision}
\begin{definition}
    \label{def:cano}
    Given a discrete-time cocycle $A$ over the driving $\phi:\Theta \to \Theta$, we say that $\Phi$, a continuous-time cocycle over the suspension flow $\phi^t:\Theta_1 \to \Theta_1$, is a \emph{canonical extension} of $A$ if
    \begin{equation}
        \Phi^n_{(\theta, 0)} = A^n_{\theta}, \quad \forall n\in \Z,
    \end{equation}
    for all $\theta \in \Theta$.
\end{definition}
\begin{remark}\label{rem:n_1_suffices}
    The two computations (\ref{eq:disc_cont_comp_1}) and (\ref{eq:disc_cont_comp_2}) show that, in order to find a canonical extension $\Phi$ of $A$, it suffices to find a continuous-time cocycle $\Phi$ that satisfies $\Phi^1_{(\theta, 0)}=A_{\theta}$ for all $\theta\in \Theta$.
    Since the flows $\Phi^t_{(\theta, 0)}$, for $t\in (0,1)$, can be rescaled in time arbitrarily, canonical extensions are, given they exist, not unique.
\end{remark}

The next theorem characterises precisely under which condition a canonical extension of a discrete-time cocycle exists.
\begin{theorem}\label{thm:nullhomotopic_extension}
    A discrete-time cocycle $A$ has a canonical extension if and only if the map
    \begin{equation}
        A: \Theta \to \SL{d},
    \end{equation}
    is nullhomotopic.
\end{theorem}
\begin{proof}
    ``$\Rightarrow$'' Assume that $A$ has a canonical extension $\Phi$. Define a map by
    \begin{nalign}
        h:\Theta\times [0,1] &\to \SL{d},\\
        (\theta, t) \hspace{6mm} &\mapsto \Phi^t_{(\theta, 0)}.
    \end{nalign}
    Since $\Phi$ is continuous in both variables, so is $h$. Hence, $h$ is a homotopy with
    \begin{equation}
        h(\theta, 0) = \Phi^0_{(\theta, 0)} = \Id, \qquad h(\theta, 1) = \Phi^1_{(\theta, 0)} = A_{\theta}.
    \end{equation}
    This shows that $A$ is nullhomotopic.\\

    ``$\Leftarrow$'' Assume that $A$ is nullhomotopic. We show that there is a homotopy $h$ from $\Id$ to $A$, with the following properties:
    \begin{enumerate}
        \item $h$ is smooth in $t$;
        \item $\partial_t h(\theta, t)$ is continuous in both variables;
        \item $\partial_t h(\theta,t) \big\vert_{t=0,1} = 0$.
    \end{enumerate}
    To show the existence of such a homotopy, we use Whitney's Approximation Theorem \cite[Theorem 6.26]{lee2012smooth}. The theorem implies that any continuous map $A:\Theta \to \SL{d}$ is homotopic to a smooth map $B\in C^\infty(\Theta ,\SL{d})$ via a homotopy $H_2$. 
    To have an explicit construction of $H_2$, we briefly sketch Lee's proof in \cite{lee2012smooth}:
    $\SL{d}$ is a smooth submanifold of $\R^{d\times d}$  without boundary. Let $U$ be a tubular neighbourhood of $\SL{d}$ in $\R^{d\times d}$ and let $r:U\to \SL{d}$ be a smooth retraction. By Whitney's Approximation Theorem for Functions \cite[Theorem 6.21]{lee2012smooth} there is a smooth map $A^*\in C^\infty(\Theta, U)$ that is sufficiently close to $A$ such that $(1-t)A(\theta) + tA^*(\theta) \in U$ for all $\theta$ and $t\in [0,1]$. Define $B=r\circ A^*\in C^\infty(\Theta ,\SL{d})$. The homotopy $H_2$ from $B$ to $A$ is given by
    \begin{equation}
        H_2(\theta, t) = r\left((1-t)A^*(\theta) + tA(\theta)\right).
    \end{equation}
    As a combination of continuous maps, $H_2$ is continuous. Since $r$ is smooth, $H_2$ is smooth in $t$. The $t$-derivative of $H_2$ is given by
    \begin{equation}
        \partial_t H_2 (\theta, t) = r'((1-t)A^*(\theta) + tA(\theta)) \cdot (A(\theta) - A^*(\theta)).
    \end{equation}
    This shows that the $t$-derivative of $H_2$ is continuous in both variables.
    
    Naturally, $B$ is still nullhomotopic. Since both $B$ and $\Id$ are smooth maps from $\Theta$ to $\SL{d}$, \cite[Theorem 6.29]{lee2012smooth} states that there is a smooth homotopy $H_1$ from $\Id$ to~$B$. Concatenating $H_1$ and $H_2$ yields a homotopy $H$ from $\Id$ to $A$ that is smooth in $t$. Since the $t$-derivatives of $H_1$ and $H_2$ are continuous in both variables, so is the $t$-derivative of $H$.
    
    Let $\sigma:[0,1] \to [0,1]$ be a smooth function with $\sigma(0) = 0$, $\sigma(1)= 1$ and $\sigma'(0)=\sigma'(1)=0$. We define a homotopy $h$ by
    \begin{equation}
        h(\theta, t) = H(\theta, \sigma(t)).
    \end{equation}
    By construction, $h$ is a homotopy from $\Id$ to $A$. Since $\sigma$ is smooth and $H$ is smooth in $t$, the homotopy $h$ is smooth in $t$. We compute the $t$-derivative of $h$ as
    \begin{equation}\label{eq:derivaitve_sigma}
        \partial_t h(\theta,t) = \sigma'(t)\partial_t H(\theta, \sigma(t)).
    \end{equation}
    This shows that the $t$-derivative of $h$ is continuous in both variables and is equal to zero for $t=0,1$. Hence, $h$ is a homotopy from $\Id$ to $A$ with the desired properties.\\

    Based on the homotopy $h$, we define a generator $G: \Theta_1 \to  \mathfrak{sl}_d(\R)$. Define
    \begin{equation}\label{eq:G_from_homotopy}
        G(\theta, r) := \partial_r h(\theta, r) \cdot h(\theta, r)^{-1},
    \end{equation}
    where $h(\theta, r)^{-1}$ is the inverse to $h(\theta, r)  \in \SL{d}$.
    To see that $G(\theta, r)\in  \mathfrak{sl}_d(\R)$ we rewrite (\ref{eq:G_from_homotopy}) as
    \begin{equation}\label{eq:G_from_homotopy_ODE}
        \partial_r h(\theta, r) = G(\theta, r) \cdot h(\theta, r).
    \end{equation}
    Since $h$ maps to $\SL{d}$, Liouville's formula guarantees $G(\theta, r)\in  \mathfrak{sl}_d(\R)$. 
    From the definition~\eqref{eq:G_from_homotopy}, we get that $G$ is continuous on $\Theta \times (0,1) \subset \Theta_1$ by continuity of $\partial_r h(\theta, r)$ and $h(\theta, r)^{-1}$. We check that $G$ is also continuous at $\Theta \times \{0,1\}\subset \Theta_1$. For $(\theta, r)\in \Theta_1$ approaching $\Theta \times \{0,1\}$, the value $r$ is approaching $0$ or $1$. From (\ref{eq:G_from_homotopy}) and (\ref{eq:derivaitve_sigma}) we can derive that $G(\theta, r) \to 0$ as $r$ approaches $0$ or $1$ since $\sigma'(r)\to 0$ while all other terms are bounded. Hence, $G$ is continuous on all of $\Theta_1$.
    
    Let $\Phi$ be the continuous-time cocycle generated by $G$. In particular, for fixed $\theta \in \Theta$, the map $t\mapsto \Phi_{(\theta, 0)}^t$ is the unique solution to the ODE, recall (\ref{eq:cont_time_cocycle_ODE}),
    \begin{equation}\label{eq:ODE_recap}
        \Phi^0_{(\theta, 0)}=\Id, \qquad \partial_t \Phi^t_{(\theta, 0)} = G(\phi^t(\theta, 0)) \cdot \Phi^t_{(\theta, 0)},
    \end{equation}
    For $t\in [0,1]$, we find $G(\phi^t(\theta, 0)) = G(\theta, t)$. Hence, (\ref{eq:G_from_homotopy_ODE}) together with the fact that $h(\theta, 0)=\Id$ shows that $t \mapsto h(\theta, t)$ solves (\ref{eq:ODE_recap}) as well for $t\in [0,1]$. By uniqueness, we find 
    \begin{equation}
        \Phi^t_{(\theta, 0)} = h(\theta, t),
    \end{equation}
    for all $t\in [0,1]$. In particular, we find
    \begin{equation}
        \Phi^1_{(\theta, 0)} = h(\theta, 1) = A_{\theta}.
    \end{equation}
    Based on Remark \ref{rem:n_1_suffices}, this shows that $\Phi$ is a canonical extension of $A$, which completes the proof.
\end{proof}

\begin{remark}
    Given any nullhomotopic cocycle $A\in C(\Theta, \SL{d})$, the existence of a tubular neighborhood $U$ and a smooth retraction $r:U\to \SL{d}$ shows that any sufficiently small perturbation of $A$ in $C(\Theta, \SL{d})$ is still nullhomotopic. With Theorem~\ref{thm:nullhomotopic_extension} this implies that the existence of a canonical extension is a robust property under small perturbations in~$C(\Theta, \SL{d})$. Conversely, a cocycle not being nullhomotopic and not having a canonical extension is robust under perturbations as well.
\end{remark}

\subsection{(Quasi-)periodic driving}
A particularly interesting case to study the existence of canonical extensions to continuous time of discrete-time cocycles is that of (quasi-)periodic driving over a circle or torus, respectively. The unit circle $S^1$ will always be identified with the interval $[0,1]$ with connected endpoints. A circle rotation is defined by the transformation
\begin{align}
\begin{split}
    \phi: S^1 &\to S^1 \\
    \theta \hspace{2mm}&\mapsto \theta + \alpha \mod 1
\end{split}
\end{align}
where $\alpha \in \R$ is the rotation constant. It is well-known that for $\alpha \in \R\setminus \mathds{Q}$, the circle rotation is uniquely ergodic with respect to Lebesgue measure on $S^1$. In that case we call $\phi$ an irrational circle rotation or quasi-periodic driving. In the case that $\alpha = \frac{p}{q} \in \mathds{Q}$, the map $\phi$ is periodic with period $q$ and we call $\phi$ periodic driving. There are infinitely many ergodic measures for periodic $\phi$, each of which consists of $q$ equally weighted delta measures in discrete points which form an orbit of~$\phi$. For the rest of this section, we work with general $\alpha \in \R$.

The flow $\phi^t$ on the mapping torus $\Theta_1$ of $S^1$ constructed from a circle rotation $\phi$ by $\alpha$ is given by
\begin{nalign}
    \phi^t: \hspace{3mm}\Theta_1 &\to \Theta_1 \\
    (\theta, r) &\mapsto (\theta + \lfloor r+t \rfloor \alpha, r+t), \mod 1.
\end{nalign}
We call $\phi^t$ the \emph{suspension flow by $\alpha$}. The mapping torus $\Theta_1$ is homeomorphic to the 2-torus $\tor = S^1 \times S^1$. Hence, the suspension flow by $\alpha$ is simply a coordinate transform of a torus rotation by $\alpha$, also known as \emph{Kronecker flow} or \emph{quasi-periodic flow},
\begin{align}
\begin{split}
    \phi^t: \hspace{3mm}\tor &\to \tor \\
    (\theta, r) &\mapsto (\theta + \alpha t, r + t)
\end{split}
\end{align}
For $\alpha \in \R \setminus \mathds{Q}$, both of these flows are uniquely ergodic with respect to Lebesgue measure on~$\Theta_1$ and $\tor$ respectively. While the flow over $\tor$ might be easier to grasp, the suspension flow is easier to work with and aligns with the previous section on continuous-time extensions. Hence, we will work with the latter and leave the torus rotation as a remark. 

The key advantage of specifically considering $\Theta=S^1$ in the context of continuous-time extensions is that the homotopy classes of $C(S^1, \SL{d})$ are well studied (cf.\ Section~\ref{sec:homotopies}).
In particular, the homotopy classes form the fundamental group $\pi_1(\SL{d})$. The fundamental group of $\SL{d}$ can be characterised by the following isomorphic relationships.
\begin{equation}
    \label{eq:SL_fundgroups}
    \pi_1(\SL{2}) \cong \Z, \qquad \pi_1(\SL{d}) \cong \Z/2\Z, \quad d\geq 3,
\end{equation}
see \cite[Chapter 1]{hall2013lie}. Since a map $S^1\to \SL{d}$ is nullhomotopic if and only if it is homotopic to a constant map, it suffices to check whether a continuous, discrete-time cocycle $A:S^1\to \SL{d}$ is homotopic to the constant identity map $\Id:S^1 \to \SL{d}$ to determine whether $A$ is nullhomotopic.

\begin{revision}While Theorem \ref{thm:nullhomotopic_extension} shows that discrete-time cocycles which are not nullhomotopic cannot be canonically extended to continuous time, a corollary of Theorem \ref{thm:nullhomotopic_extension} shows that for any discrete-time cocycle $A$ in at least 3 dimensions there is a continuous-time cocycle extension $\Phi$ upon halving the rotation constant $\alpha$.\end{revision} 
\begin{corollary}\label{crl:n_geq_3_extension}
    Let $d \geq 3$. Let $A$ be a discrete-time cocycle in $\SL{d}$ over a circle rotation by $\alpha$. There is a continuous-time cocycle $\Phi$ over the suspension flow by $\frac{1}{2}\alpha$ that satisfies
    \begin{equation}
        \Phi^n_{(\theta, 0)} = A^n_{2\theta}, \quad\forall n\in \Z,
    \end{equation}
    for all $\theta \in S^1$.
\end{corollary}
\begin{proof}
    From $A$ we construct another discrete-time cocycle $B$ over a circle rotation by $\frac{1}{2}\alpha$ via
    \begin{equation}
        B_{\theta} = A_{2\theta}.
    \end{equation}
    Since the rotation speed of the driving of $B$ is half of that of $A$ we find
    \begin{align}\label{eq:A_B_equivalence}
    \begin{split}
        B_{\theta}^n &= B_{\theta+(n-1)\frac{1}{2}\alpha}\cdot\hdots \cdot B_{\theta + \frac{1}{2}\alpha}\cdot B_{\theta}\\
        &= A_{2\theta+(n-1)\alpha}\cdot\hdots \cdot A_{2\theta + \alpha}\cdot A_{2\theta} \\
        &= A^n_{2\theta}.
    \end{split}
    \end{align}
    Observe that, as a loop, $B$ is the concatenation of $A$ with itself (cf.~\eqref{eq:concatenation}). We find
    \begin{equation}\label{eq:self_concatenation}
        [B] = [A] \circ [A].
    \end{equation}
    For $d\geq 3$, the fundamental group of $\SL{d}$ is isomorphic to $\Z/2\Z$. Hence, (\ref{eq:self_concatenation}) shows that $B$ is nullhomotopic. Theorem \ref{thm:nullhomotopic_extension} guarantees the existence of a continuous-time cocycle $\Phi$ over a suspension flow by $\frac{1}{2}\alpha$ that satisfies
    \begin{equation}
        \Phi^n_{(\theta, 0)} = B^n_{\theta}, \quad\forall n\in \Z.
    \end{equation}
    This, together with (\ref{eq:A_B_equivalence}) completes the proof.
\end{proof}

In the case of $d=2$, we can add an auxiliary neutral dimension to represent any discrete-time cocycle in continuous-time. To distinguish the behaviour in different dimensions, we write an element $x\in \R^{3}$ as $x=(v,w)$ where $v\in \R^2$ and $w\in \R$.
\begin{corollary}\label{crl:n_2_extension}
    Let $A$ be a discrete-time cocycle in $\SL{2}$ over a circle rotation by $\alpha$. There is a continuous-time cocycle $\Phi$ in $\SL{3}$ over a suspension flow by $\frac{1}{2}\alpha$ that satisfies
    \begin{equation}
        \Phi^n_{(\theta, 0)}(v,w) = (A^n_{2\theta}v, w), \quad\forall n\in \Z,
    \end{equation}
    for all $\theta \in S^1$, where $v\in \R^2$ and $w\in \R$.
\end{corollary}
\begin{proof}
    Define the discrete-time cocycle $B$ over a circle rotation by $\alpha$ via
    \begin{equation}
        B_{\theta} = \begin{pmatrix}
        A_{\theta} & \rvline & \begin{matrix} 0 \\ 0\end{matrix} \\ \hline
        \begin{matrix} 0 & 0\end{matrix} & \rvline & 1
    \end{pmatrix} \in \SL{3}.
    \end{equation}
    Applying Corollary \ref{crl:n_geq_3_extension} to $B$ yields a continuous-time cocycle $\Phi$ over a suspension flow by $\frac{1}{2}\alpha$ that satisfies
    \begin{equation}\label{eq:n=2_extension}
        \Phi^n_{(\theta, 0)} = B_{2\theta}^n = \begin{pmatrix}
        A^n_{2\theta} & \rvline & \begin{matrix} 0 \\ 0\end{matrix} \\ \hline
        \begin{matrix} 0 & 0\end{matrix} & \rvline & 1
    \end{pmatrix}, \quad\forall n\in \Z.
    \end{equation}
    This completes the proof.
\end{proof}

\begin{revision}
\begin{remark}\label{rem:preserve_hyper}
    Canonical extensions as well as the non-canonical extensions of Corollary \ref{crl:n_geq_3_extension} and Corollary \ref{crl:n_2_extension} of a discrete-time cocycle $A$ preserve the Furstenberg--Kesten limits and the Lyapunov exponents of~$A$. In the case of canonical extensions, as well as the case \mbox{$d\geq 3$}, this can be verified by a straightforward computation. The case $d=2$ requires a little bit more effort due to the added dimension, which contributes an additional zero Lyapunov exponent.
    Additionally, a canonical extension $\Phi$ is \mbox{(non-)uniformly} hyperbolic if and only if $A$ is \mbox{(non-)uniformly} hyperbolic.
\end{remark}
\end{revision}

\begin{revision}
    \section{A continuous-time non-uniformly hyperbolic cocycle}\label{sec:construction}
    In this section we construct an example of a continuous-time cocycle in $\SL{2}$ over a quasi-periodic torus rotation, which is non-uniformly hyperbolic. We construct such an example as the canonical extension of a nullhomotopic discrete-time cocycle in $\SL{2}$ which is non-uniformly hyperbolic. Let $\phi^t$ be an irrational circle rotation by $\alpha$. The discrete-time cocycle we construct is of the form
    \begin{equation}\label{eq:A_from_f}
        A_\theta = \Lambda \cdot R_{f(\theta)},
    \end{equation}
    where $f: S^1 \to S^1$ is differentiable, $\Lambda$ is the diagonal matrix with entries $\lambda, \lambda^{-1}$ and $R_{f(\theta)} \in \SL{2}$ is the rotation matrix by angle $2\pi f(\theta)$, i.e.
    \begin{equation}
         \Lambda = \begin{pmatrix} \lambda & 0 \\ 0 & \lambda^{-1} \end{pmatrix}, \quad R_{f(\theta)} = \begin{pmatrix} \cos(2\pi f(\theta)) & \sin(2\pi f(\theta)) \\ -\sin(2\pi f(\theta)) & \cos(2\pi f(\theta)) \end{pmatrix}.
    \end{equation}
    For the identity function $f(\theta) = \theta$, the resulting cocycle $A_\theta$ has been studied by Herman \cite{herman1981construction}. It was shown that for suited rotation constants $\alpha$, this cocycle is non-uniformly hyperbolic. However, the following lemma shows that for $f$ the identity function, the resulting cocycle is not nullhomotopic and, hence, does not have a canonical extension. 
    \begin{lemma}\label{lem:f_nullhomotopic}
        The cocycle $A_\theta$ as defined in~\eqref{eq:A_from_f} is nullhomotopic if and only if the map $f$ is homotopic to a constant.
    \end{lemma}
    \begin{proof}
        Assume that $f:S^1 \to S^1$ is homotopic to a constant, i.e.~there is a homotopy $h:S^1 \times [0,1] \to S^1$ such that
        \begin{equation}
            h(\theta, 0) = 0, \qquad h(\theta, 1) = f(\theta).
        \end{equation}
        Then, $A:S^1 \to \SL{2}$ is nullhomotopic via the homotopy
        \begin{nalign}
            H:S^1\times [0,1] &\to \SL{2},\\
            (\theta, t) \hspace{6mm} &\mapsto \Lambda \cdot R_{h(\theta, t)}.
        \end{nalign}

        Now, assume that the cocycle $A$ is nullhomotopic, i.e.~that $A$ is homotopic to the identity via a homotopy $H$. For $v \in \R^2 \setminus \{0\}$, let $\overline{v}\in \mathds{P}^1$ be the projectivisation of $v$. Let $e_1, e_2 \in \R^2$ be the unit vectors. We identify the projective space $\mathds{P}^1$ with $S^1$ such that $\overline{e_1} = 0$ and $\overline{e_2} = \frac{1}{2}$. We construct a map
        \begin{nalign}
                h:S^1\times [0,1] &\to S^1,\\
                (\theta, t) \hspace{6mm} &\mapsto \overline{H(\theta, t) e_1}.
        \end{nalign}
        Since $H(\theta, 0)= \Id$, we get $h(\theta, 0) = 0$. On the other hand $H(\theta, 1) = A_\theta = \Lambda\cdot R_{f(\theta)}$. Noting that rotation numbers expressed in the projective-space $\mathds{P}^1$ are doubled from those in $\R^2$, we find $\overline{R_{f(\theta)} e_1} =2f(\theta)$, while $\Lambda$ can not rotate any element in $\mathds{P}^1$ by more than $\frac{1}{2}$. We conclude that $\Lambda$ has no influence on the winding number of~$h(\cdot, 1)$. Hence, the winding number of $h(\cdot, 1)$ is twice the winding number of~$f$. Since $h(\cdot, 1)$ is homotopic to a constant, it has winding number~0. Therefore, $f$ must also have winding number 0, which is equivalent to $f$ being homotopic to a constant. This completes the proof.
    \end{proof}
    It remains to find a map $f$ that is homotopic to a constant and a rotation constant $\alpha$, such that the cocycle, as defined in~\eqref{eq:A_from_f}, is non-uniformly hyperbolic. We choose a map $f$ with the following properties
    \begin{enumerate}[(i)]
        \item $f$ is nullhomopic;
        \item $f^{-1}(\frac{1}{2})$ is non-empty;
        \item $\partial_\theta f(\theta) \neq 0$ for all $\theta \in f^{-1}(\frac{1}{2})$.
    \end{enumerate}
    Let $A$ be the cocycle as defined in~\eqref{eq:A_from_f} with a map $f$ that has the three properties listed above, e.g.~$f(\theta) = \frac12 + \varepsilon \sin(2\pi \theta)$ for $\varepsilon \in (0,1)$. Given that $\lambda$ is sufficiently large, a result by Young \cite[Theorem 1]{young1997lyapunov} states that there is $\alpha\in \R\setminus \mathds{Q}$ such that $A$ has positive Lyapunov exponent as a cocycle over a circle rotation by $\alpha$. We note that properties (ii) and (iii) are equivalent to condition (T1) in the paper by Young. Additionally, \cite[Remark 1]{young1997lyapunov} states that the cocycle $A$ cannot be uniformly hyperbolic. Hence, $A$ must be non-uniformly hyperbolic. Lemma \ref{lem:f_nullhomotopic} shows that $A$ is nullhomotopic and, therefore, has a canonical extension $\Phi$, as shown by Theorem \ref{thm:nullhomotopic_extension}. By Remark \ref{rem:preserve_hyper}, the continuous-time cocycle $\Phi$ is non-uniformly hyperbolic. 
    
    Note that for $f(\theta) = \frac12 + \varepsilon \sin(2\pi \theta)$ and $\varepsilon$ approaching 0, we can find non-uniformly hyperbolic and nullhomotopic discrete-time cocycles that are arbitrarily close to an autonomous cocycle.
\end{revision}

\section*{Acknowledgements}

This research has been partially supported by Deutsche Forschungsgemeinschaft (DFG) through grant CRC 1114 ``Scaling Cascades in Complex Systems'', Project Number 235221301, Project A08 ``Characterization and prediction of quasi-stationary atmospheric states’’. M.E. additionally acknowledges the Einstein
Foundation (IPF-2021-651) and Germany’s Excellence Strategy—The Berlin Mathematics Research Center MATH+ (EXC-2046/1, project ID: 390685689) for support.
We thank Alex Blumenthal for insightful discussions and the anonymous reviewer for their excellent suggestions.

\bibliography{bibliography.bib}
\bibliographystyle{myalpha}

\end{document}